\documentclass[a4paper,english]{amsart}
\usepackage[latin9]{inputenc}
\usepackage{amsmath}

\usepackage{bbm}
\usepackage{babel}
\usepackage{verbatim}
\usepackage{mathtools}
\usepackage{amsthm}
\usepackage{amstext}
\usepackage{amssymb}
\usepackage{amsfonts}
\usepackage{esint}
\usepackage[unicode=true,
bookmarks=false,
breaklinks=false,pdfborder={0 0 1},backref=false,colorlinks=false]
{hyperref}
\hypersetup{
	colorlinks,linkcolor=blue,anchorcolor=blue,citecolor=blue}
\usepackage{breakurl}
\usepackage{mathrsfs}
\usepackage{graphicx}
\usepackage{tikz}
\usetikzlibrary{matrix,arrows,decorations.pathmorphing}
\usepackage{amsmath}

\makeatletter

\numberwithin{equation}{section}
\numberwithin{figure}{section}
\usepackage{enumitem}		% customizable list environments

\theoremstyle{plain}
\newtheorem{thm}{Theorem}[section]
\theoremstyle{plain}
\newtheorem{prop}[thm]{Proposition}
\theoremstyle{remark}
\newtheorem{rem}[thm]{Remark}

\theoremstyle{plain}

\theoremstyle{plain}

\theoremstyle{plain}
\newtheorem{lem}[thm]{Lemma}
\theoremstyle{definition}
\newtheorem{example}[thm]{Example}
\theoremstyle{definition}
\newtheorem{defn}[thm]{Definition}

\begin{document}
	
\title{Idempotent states on Sekine quantum groups}

\author{Haonan Zhang}

\address{Laboratoire de Math\'ematiques, Universit\'e de Bourgogne Franche-Comt\'e, 25030 Besan\c con, France and Institute of Mathematics, Polish Academy of Sciences, ul. \'Sniadeckich 8, 00-956 Warszawa, Poland}

\email{haonan.zhang@edu.univ-fcomte.fr}

\subjclass[2010]{Primary: 20G42, 15A24.  Secondary: 60B15.}

\keywords{finite quantum groups, Sekine quantum groups, idempotent states, random walks}

\maketitle
\begin{abstract}
	Sekine quantum groups are a family of finite quantum groups. The main result of this paper is to compute all the idempotent states on Sekine quantum groups, which completes the work of Franz and Skalski. This is achieved by solving a complicated system of equations using linear algebra and basic number theory. From this we discover a new class of non-Haar idempotent states. The order structure of the idempotent states on Sekine quantum groups is also discussed. Finally we give a sufficient condition for the convolution powers of states on Sekine quantum group to converge. 
\end{abstract}

\section*{Introduction}
On a locally compact group, probability measures which are idempotent with respect to convolution arise as Haar measures on its compact subgroups. This was first proved by Kawada and It\^o \cite[Theorem 3]{KawadaIto}. See also \cite{Heyer} and references therein. Such a result fails for quantum groups. Namely, idempotent states on locally compact quantum groups are not necessarily of Haar type -- in other words, they are not \emph{Haar idempotent states}. Pal \cite{Pal} gave the first counterexample on a 8-dimensional Kac-Paljutkin quantum group. Indeed, he showed the existence of two idempotent states on this quantum group whose null-spaces are not self-adjoint. Such counterexamples are called \emph{non-Haar idempotent states}. Later on more non-Haar idempotent states have been found, see \cite{FSonidempotentstatesonquantumgroups}. Even simpler counterexamples can be easily constructed on group algebras of finite non-commutative groups, as pointed out in \cite{FSonidempotentstatesonquantumgroups}. 

Idempotent states also arise as limits in ergodic theorems for random walks \cite{FSergodicproperties}. The Ces\`aro limit of convolution powers of a state on a compact quantum group gives an idempotent state. It becomes the Haar state if the original state is suitably chosen (for example faithful). This is how Woronowicz constructed the Haar state \cite{CompactquantumgroupsWoronowicz}. We will say a few words on the convergence of convolution powers on \emph{Sekine quantum groups}, which are main examples of quantum groups we focus on in this paper.

For those who are interested in further applications of idempotent states, we refer to \cite{Hopfimages} and \cite{Poissonboundaries} for connections of idempotent states with Hopf images and Poisson boundaries, respectively. For more characterizations of idempotent states on compact quantum groups, see \cite{FSonidempotentstatesonquantumgroups,Newcharacterisationofidempotentstates}. 

Despite much progress on idempotent states, concrete forms of idempotent states on compact quantum groups have only been computed in a few examples. On classical (commutative) compact quantum groups, idempotent states, due to Kawada and It\^o \cite{KawadaIto}, are all Haar idempotents. On cocommutative quantum group $C^*(\Gamma)$ with $\Gamma$ a finite group, idempotent states are associated with subgroups of $\Gamma$, while Haar idempotent states correspond to the normal subgroups \cite{FSonidempotentstatesonquantumgroups}. Idempotent states on compact quantum groups $U_q(2),SU_q(2)$ and $SO_q(3)$ were determined by Franz, Skalski and Tomatsu \cite{FSTidempotentstatesonSU_2}. In \cite{FSonidempotentstatesonquantumgroups} Franz and Skalski also gave some examples of idempotent states on the so-called \emph{Sekine quantum groups}, including all the Haar idempotent states and some classes of non-Haar ones. Sekine quantum groups form a class of finite quantum groups of Kac-Paljutkin type introduced by Sekine \cite{Sekine}. The main purpose of this paper is to determine all the idempotent states on Sekine quantum groups, completing the work in \cite{FSonidempotentstatesonquantumgroups}. This is achieved in Theorem \ref{thm:non-Haar} by solving a complicated system of equations in Lemma \ref{characterization of idempotent states on Sekine quantum groups}. 

The plan of this paper is as follows. In Section 1 we recall the definition of Sekine quantum groups and describe their representation theory. Then in Section 2 we give the calculations of idempotent states in detail. It turns out that the set of idempotent states, other than Haar state, can be divided into three disjoint parts. In Section 3 we discuss the order structure of idempotent states. The last section a sufficient condition for convolution powers of states on Sekine quantum groups to converge is given.

\section{Sekine quantum groups}
Although the main subject of this paper lies in finite quantum groups, we begin with the discussion of compact quantum group. We refer to \cite{CompactquantumgroupsWoronowicz} and \cite{CompactquantumgroupsVanDaele} for more details on compact quantum groups.

\subsection{Compact quantum groups}
The following definition of compact quantum group is due to Woronowicz \cite{CompactquantumgroupsWoronowicz}.

\begin{defn}
	Let $A$ be a unital $C$*-algebra. If there exists a unital *-homomorphism $\Delta:A\to A\otimes A$ such that 
	\begin{enumerate}
		\item $(\Delta\otimes\iota)\Delta=(\iota\otimes\Delta)\Delta$;
		\item $\overline{\text{Lin}}\{\Delta(a)(1\otimes b):a,b\in A\}=\overline{\text{Lin}}\{\Delta(a)(b\otimes1):a,b\in A\}=A\otimes A$,
	\end{enumerate}
	then the pair $\mathbb{G}=(A,\Delta)$ is called a \emph{compact quantum group} and $\Delta$ is called the \emph{comultiplication} on $A$. Here $\iota$ denotes the identity map and $\otimes$ denotes the spatial tensor product of $C$*-algebras. Sometimes we write the underlying $C$*-algebra $A$ as the algebra of continuous functions on a quantum group $\mathbb{G}$, using the notation $A=C(\mathbb{G})$.
\end{defn}

Any compact quantum group $\mathbb{G}=(A,\Delta)$ admits a unique \emph{Haar state} $h$, which is a state (a positive normalised functional) on $A$ verifying the condition
\[
(h\otimes\iota)\Delta(a)=h(a)1=(\iota\otimes h)\Delta(a),~~a\in A.
\]
Consider an element $u\in A\otimes B(H)$ with $H$ a Hilbert space such that $\dim H=n$. By identifying $A\otimes B(H)$ with $\mathbb{M}_n(A)$ we can write $u=[u_{ij}]_{i,j=1}^{n}$. Such a matrix $u$ is called a \emph{n-dimensional representation} of $\mathbb{G}$ if we have
\[
\Delta(u_{ij})=\sum_{k=1}^{n}u_{ik}\otimes u_{kj},~~i,j=1,\dots,n.
\]
A representation $u$ is called \emph{unitary} if it is unitary as an element of $M_n(A)$, and \emph{irreducible} if the only matrices $T\in\mathbb{M}_n(\mathbb{C})$ such that $uT=Tu$ are multiples of identity matrix. Two representations $u,v\in\mathbb{M}_n(A)$ are said to be \emph{equivalent} if there exists an invertible matrix $T\in\mathbb{M}_n(\mathbb{C})$ such that $Tu=vT$. Denote by $\text{Irr}(\mathbb{G})$ the set of equivalence classes of irreducible unitary representations of $\mathbb{G}$. For each $\alpha\in\text{Irr}(\mathbb{G})$, denote by $u^\alpha=[u^\alpha_{ij}]_{i,j=1}^{n_\alpha}\in A\otimes B(H_\alpha)$ a representative of the class $\alpha$, where $H_\alpha$ is the finite dimensional Hilbert space on which $u^\alpha$ acts. In the sequel we write $n_\alpha=\dim H_\alpha$. Denote $\text{Pol}(\mathbb{G})=\text{span}\{u^\alpha_{ij}:1\leq i,j\leq n_\alpha, \alpha\in\text{Irr}(\mathbb{G})\}$. It is a dense subalgebra of $A$.
And it is known that $\text{Pol}(\mathbb{G})$ possesses a \emph{counit}, which is a *-homomorphism $\epsilon:\text{Pol}(\mathbb{G})\to\mathbb{C}$ such that
\[
(\epsilon\otimes\iota)\Delta(a)=a=(\iota\otimes\epsilon)\Delta(a),~~a\in\text{Pol}(\mathbb{G}).
\]

A compact quantum group $\mathbb{G}=(A,\Delta)$ is a \emph{finite quantum group} if the underlying $C$*-algebra $A$ is finite-dimensional. In this case $A=\text{Pol}(\mathbb{G})$. And the Haar state $h$ on a finite quantum group is also a trace \cite{Haarmeasureonfinitequantumgroups}. That is, $h(ab)=h(ba), a,b\in A.$

\subsection{Sekine quantum groups}
Sekine \cite{Sekine} introduced a family of finite quantum groups,  referred as \emph{Sekine quantum groups}, arising as bicrossed products of classical cyclic groups with the matched pair being $\mathbb{Z}_2$ and $\mathbb{Z}_k\times\mathbb{Z}_k$ \cite{VaesVainerman}. $\mathbb{Z}_2$ acts on $\mathbb{Z}_k\times\mathbb{Z}_k$ by permutation. We follow the notations in \cite{FSonidempotentstatesonquantumgroups} here. 

\begin{defn}
	Fix $k\geq 2$ an integer. Let $\eta$ be a primitive $k$-th root of $1$, say, $\eta=e^{\frac{2\pi i}{k}}$, and let $\mathbb{Z}_k:=\{0,1,\dots,k-1\}$ be the cyclic group with order $k$. Set 
	\[
	\mathcal{A}_k:=\bigoplus_{i,j\in \mathbb{Z}_k}\mathbb{C}d_{i,j}\oplus M_k(\mathbb{C}).
	\]
	Denote by $\{e_{i,j}:i,j\in\mathbb{Z}_k\}$ the matrix units of $M_k(\mathbb{C})$. The comultiplication on $\mathcal{A}_k$ is defined through:
	
	\begin{equation}\label{Def of Sekine 1}
	\Delta_k(d_{i,j}):=\sum_{m,n\in\mathbb{Z}_k}d_{m,n}\otimes d_{i-m,j-n}+\frac{1}{k}\sum_{m,n\in\mathbb{Z}_k}\eta^{i(m-n)}e_{m,n}\otimes e_{m+j,n+j},   
	\end{equation}
	
	\begin{equation}\label{Def of Sekine 2}
	\Delta_k(e_{i,j}):=\sum_{m,n\in \mathbb{Z}_k}\eta^{m(i-j)}d_{-m,-n}\otimes e_{i-n,j-n}+\sum_{m,n\in \mathbb{Z}_k}\eta^{m(j-i)}e_{i-n,j-n}\otimes d_{m,n},
	\end{equation}
	for $i,j\in \mathbb{Z}_k$. Then the pair $(\mathcal{A}_k,\Delta_k)$ forms a finite quantum group, called a \emph{Sekine quantum group}.
\end{defn}

\subsection{Representation of Sekine quantum groups}
We refer to \cite{JPMcCarthy} for more discussions on representation theory of Sekine quantum groups. Let $p,q\in\mathbb{Z}_k$. Then from \eqref{Def of Sekine 1} it follows that 
\begin{align*}
\sum_{i,j\in\mathbb{Z}_k}\eta^{ip+jq}\Delta_k(d_{i,j})
=&\sum_{m,n\in\mathbb{Z}_k}\eta^{mp+nq}d_{m,n}\otimes \sum_{i,j\in\mathbb{Z}_k}\eta^{ip+jq}d_{i,j}\\
&+\sum_{m\in\mathbb{Z}_k}\eta^{-mq}e_{m,m+p}\otimes\sum_{j\in\mathbb{Z}_k}\eta^{jq}e_{j,j+p},
\end{align*}
and from \eqref{Def of Sekine 2} it follows that
\begin{align*}
\sum_{i\in\mathbb{Z}_k}\eta^{iq}\Delta_k(e_{i,i+p})
=&\sum_{m,n\in \mathbb{Z}_k}\eta^{mp-nq}d_{m,n}\otimes\sum_{i\in \mathbb{Z}_k}\eta^{iq}e_{i,i+p}\\
&+\sum_{i\in\mathbb{Z}_k}\eta^{iq}e_{i,i+p}\otimes\sum_{m,n\in\mathbb{Z}_k} \eta^{mp+nq}d_{m,n}.
\end{align*}
Set $\rho_{p,q}:=\sum_{m,n\in\mathbb{Z}_k}\eta^{mp+nq}d_{m,n}$ and $\sigma_{p,q}:=\sum_{i\in \mathbb{Z}_k}\eta^{iq}e_{i,i+p}$ for all $p,q\in\mathbb{Z}_k$. Then the equations above can be rephrased as 
\begin{equation*}
\Delta_k(\rho_{p,q})=\rho_{p,q}\otimes\rho_{p,q}+\sigma_{p,-q}\otimes\sigma_{p,q},
\end{equation*}
\begin{equation*}
\Delta_k(\sigma_{p,q})=\rho_{p,-q}\otimes\sigma_{p,q}+\sigma_{p,-q}\otimes\rho_{p,q},
\end{equation*}
for all $p,q\in\mathbb{Z}_k$. This yields directly that for any $p,q\in\mathbb{Z}_k$,
\begin{equation*}
\pi_{p,q}:=\begin{pmatrix*}
\rho_{p,q} & \sigma_{p,-q}\\
\sigma_{p,q} & \rho_{p,-q}
\end{pmatrix*}
\end{equation*}
is a representation of $(\mathcal{A}_k,\Delta_k)$. Moreover, it is also unitary. To see this, note by definitions of $\rho_{p,q}$ and $\sigma_{r,s}$ that
\[
\rho_{p,q}^*=\rho_{-p,-q},~~\sigma_{r,s}^*=\eta^{rs}\sigma_{-r,-s},
\]
\[\rho_{p,q}\rho_{p',q'}=\rho_{p+p',q+q'},~~\rho_{p,q}\sigma_{r,s}=\sigma_{r,s}\rho_{p,q}=0,~~\sigma_{r,s}\sigma_{r',s'}=\eta^{rs'}\sigma_{r+r',s+s'},
\]
where $p,q,p',q',r,s,r',s'\in\mathbb{Z}_k$. Then 
\begin{equation*}
\pi_{p,q}^*\pi_{p,q}
=\begin{pmatrix*}
\rho_{-p,-q} & \eta^{pq}\sigma_{-p,-q}\\
\eta^{-pq}\sigma_{-p,q} & \rho_{-p,q}
\end{pmatrix*}
\begin{pmatrix*}
\rho_{p,q} & \sigma_{p,-q}\\
\sigma_{p,q} & \rho_{p,-q}
\end{pmatrix*}
=\begin{pmatrix*}
\rho_{0,0}+\sigma_{0,0} & 0\\
0 & \rho_{0,0}+\sigma_{0,0}
\end{pmatrix*},
\end{equation*}
while $\rho_{0,0}+\sigma_{0,0}=\sum_{p,q\in\mathbb{Z}_k}d_{p,q}+\sum_{r\in\mathbb{Z}_k}e_{r,r}=1_{\mathcal{A}_k}$ is the unit element. So $\pi_{p,q}^*\pi_{p,q}=\text{id}$. Similarly, we have $\pi_{p,q}\pi_{p,q}^*=\text{id}$.

When $q=0$, $\pi_{p,q}$ is unitarily equivalent to 
\begin{equation*}
\begin{pmatrix*}
\frac{1}{\sqrt{2}} & \frac{1}{\sqrt{2}}\\
\frac{1}{\sqrt{2}}& -\frac{1}{\sqrt{2}}
\end{pmatrix*}
\begin{pmatrix*}
\rho_{p,0} & \sigma_{p,0}\\
\sigma_{p,0} & \rho_{p,0}
\end{pmatrix*}
\begin{pmatrix*}
\frac{1}{\sqrt{2}} & \frac{1}{\sqrt{2}}\\
\frac{1}{\sqrt{2}}& -\frac{1}{\sqrt{2}}
\end{pmatrix*}
=\begin{pmatrix*}
\rho_{p,0}+\sigma_{p,0} & 0\\
0& \rho_{p,0}-\sigma_{p,0}
\end{pmatrix*}.
\end{equation*}
For the same reason, if $k$ is even, $\pi_{p,k/2}$ is unitarily equivalent to 
\begin{equation*}
\begin{pmatrix*}
\frac{1}{\sqrt{2}} & \frac{1}{\sqrt{2}}\\
\frac{1}{\sqrt{2}}& -\frac{1}{\sqrt{2}}
\end{pmatrix*}
\begin{pmatrix*}
\rho_{p,k/2} & \sigma_{p,k/2}\\
\sigma_{p,k/2} & \rho_{p,k/2}
\end{pmatrix*}
\begin{pmatrix*}
\frac{1}{\sqrt{2}} & \frac{1}{\sqrt{2}}\\
\frac{1}{\sqrt{2}}& -\frac{1}{\sqrt{2}}
\end{pmatrix*}
=\begin{pmatrix*}
\rho_{p,k/2}+\sigma_{p,k/2} & 0\\
0& \rho_{p,k/2}-\sigma_{p,k/2}
\end{pmatrix*}.
\end{equation*} 
In such cases $\pi_{p,q}$ can be decomposed into, up to equivalence, two one-dimensional unitary irreducible representations. So we have obtained $2k$ one-dimensional representations when $k$ is odd and $4k$ one-dimensional representations when $k$ is even. 

Moreover, $\pi_{p,q}$ is unitarily equivalent to $\pi_{p,-q}$ since
\begin{equation*}
\begin{pmatrix*}
0 & 1\\
1& 0
\end{pmatrix*}
\begin{pmatrix*}
\rho_{p,q} & \sigma_{p,-q}\\
\sigma_{p,q} & \rho_{p,-q}
\end{pmatrix*}
\begin{pmatrix*}
0 & 1\\
1& 0
\end{pmatrix*}
=\begin{pmatrix*}
\rho_{p,-q} & \sigma_{p,q}\\
\sigma_{p,-q} & \rho_{p,q}
\end{pmatrix*}.
\end{equation*}
It is known \cite{JPMcCarthy} that $\{\pi_{p,q}:1\leq q\leq [\frac{k-1}{2}]k\}$ are pairwise inequivalent two-dimensional irreducible representations. Note that although this was pointed out only for $k$ odd in \cite{JPMcCarthy}, it holds also for $k$ even, following a similar argument. 

Hence, up to equivalence, $(\mathcal{A}_k,\Delta_k)$ has $2k$ one dimensional unitary irreducible representations and $\frac{k(k-1)}{2}$ two-dimensional irreducible representations when $k$ is odd, $4k$ one dimensional unitary irreducible representations and $\frac{k(k-2)}{2}$ two-dimensional irreducible representations when $k$ is even. These are the only irreducible representations. Indeed, one can check this by verifying the dimension.

We end this section by introducing the \emph{Fourier transform} of linear functionals $\mu$ on $\in\mathcal{A}_k$, denoted by $\hat{\mu}$, at $\pi_{p,q}$:
\begin{equation*}
\hat{\mu}(\pi_{p,q})=\begin{pmatrix*}
\mu(\rho_{p,q}) & \mu(\sigma_{p,-q})\\
\mu(\sigma_{p,q}) & \mu(\rho_{p,-q})
\end{pmatrix*},~~p,q\in\mathbb{Z}_k.
\end{equation*}
It is easy to see that for any functionals $\mu,\nu$ on $\mathcal{A}_k$
\begin{equation}\label{Fourier transform}
\widehat{\mu\star\nu} (\pi_{p,q})=\hat{\mu}(\pi_{p,q})\hat{\nu}(\pi_{p,q}),~~p,q\in\mathbb{Z}_k,
\end{equation}
where $\mu\star\nu:=(\mu\otimes\nu)\Delta$ denotes the convolution of $\mu$ and $\nu$.

\section{Idempotent states on Sekine quantum groups}
For a compact quantum group $\mathbb{G}=(A,\Delta)$, denote by $A'$ the set of all linear functionals on $A$. Then for $\mu,\nu\in A'$ we can define the \emph{convolution} of $\mu$ and $\nu$, which we have seen earlier, as a linear functional on $A$ given by the formula
\[
\mu\star\nu:=(\mu\otimes\nu)\Delta.
\]
A state $\mu$ on $A$ is called an \emph{idempotent state} if $\mu\star\mu=\mu$. Denote the class of all idempotent states on $\mathbb{G}=(A,\Delta)$ by $\text{Idem}(\mathbb{G})$, or $\text{Idem}(A)$. Idempotent states on compact quantum groups have been characterized in different ways \cite{FSonidempotentstatesonquantumgroups,Newcharacterisationofidempotentstates}.

\begin{example}[Commutative case]
	If $A$ is commutative, then $\mathbb{G}$ is isomorphic to $(C(G),\Delta)$, where $C(G)$ denotes the set of continuous functions on a compact group $G$ and $\Delta$ is a comultiplication on $C(G)$ given by 
	\[
	\Delta(f)(s,t)=f(st),~~s,t\in G.
	\]
	In this case idempotent state on $C(G)$ arises as idempotent probability measure on $G$, which, by Kawada and It\^o's classical theorem, arises as the Haar measure on a compact subgroup of $G$. 
\end{example}

Let's recall the notion of a quantum subgroup here.

\begin{defn}
	If $(A,\Delta_A),(B,\Delta_B)$ are compact quantum groups and $\pi_B:A\to B$ is a surjective unital *-homomorphism such that $\Delta_B\circ\pi_B=(\pi_B\otimes\pi_B)\circ\Delta_A$, then $(B,\Delta_B)$ is called a \emph{quantum subgroup} of $(A,\Delta_A)$.
\end{defn} 

Let $h_B$ be the Haar state on $B$, then $h_B\circ\pi_B$ is an idempotent state on $A$. If an idempotent state $\phi$ on $A$ arises in this way, say, $\phi=h_B\circ\pi_B$ for some quantum subgroup $(B,\Delta_B)$ with Haar measure $h_B$ of $A$, then it is called a \emph{Haar idempotent state}. Otherwise, it is called a \emph{non-Haar idempotent state}. So Kawada and It\^o's theorem tells us that in the commutative case, all idempotent states are Haar idempotent states.

The existence of non-Haar idempotent state was first proved by Pal \cite{Pal} on a 8-dimensional Kac-Paljutkin quantum group. Even simpler examples come from cocommutative finite quantum groups \cite{FSonidempotentstatesonquantumgroups}.

\begin{example}[Cocommutative case]
	A finite quantum group $\mathbb{G}=(A,\Delta)$ is said to be \emph{cocommutative} if $\Pi\Delta=\Delta$, where $\Pi$ denotes the usual tensor flip on $A\otimes A$. Then $A$ is isomorphic to the group algebra $C^*(\Gamma)$ with $\Gamma$ a finite discrete group. Then there is a one-one correspondence between idempotent states on $A$ and subgroups of $\Gamma$. Moreover, there is a one-one correspondence between Haar idempotent states on $A$ and normal subgroups of $\Gamma$. So from a non-normal subgroup of $\Gamma$ one can construct a non-Haar idempotent.
\end{example}

More examples can be found on Sekine quantum groups \cite{FSonidempotentstatesonquantumgroups}. Indeed, a small class of non-Haar idempotent states on Sekine quantum groups $\mathcal{A}_k$ was  given in \cite[Proposition 6.6]{FSonidempotentstatesonquantumgroups}. See \eqref{FS nonHaar example 1} \eqref{FS nonHaar example 2} below for details.

Now we consider all the idempotent states on $\mathcal{A}_k$. Fix $k\geq 2$. On $\mathcal{A}'_k$ there is a natural basis:
\[
\widetilde{d}_{i,j}(d_{m,n}):=\delta^i_m\delta^j_n,~~ \widetilde{d}_{i,j}(e_{r,s}):=0,~~
\widetilde{e}_{i,j}(d_{m,n}):=0,~~\widetilde{e}_{i,j}(e_{r,s}):=\delta^i_r\delta^j_s,
\]
where $i,j,m,n,r,s\in\mathbb{Z}_k$. Using this basis, the explicit formula for the convolution of two linear functionals on $\mathcal{A}_k$ is given through the following lemma (See the discussion before \cite[Lemma 6.4]{FSonidempotentstatesonquantumgroups}).

\begin{lem}\label{main lemma}
	For $\mu=\sum_{i,j\in \mathbb{Z}_k}\alpha_{i,j}\widetilde{d}_{i,j}+\sum_{r,s\in\mathbb{Z}_k}\kappa_{r,s}\widetilde{e}_{r,s},
	\nu=\sum_{i,j\in \mathbb{Z}_k}\beta_{i,j}\widetilde{d}_{i,j}+\sum_{r,s\in\mathbb{Z}_k}\omega_{r,s}\widetilde{e}_{r,s}\in\mathcal{A}'_k$, we have 
	\[
	\mu\star\nu=\sum_{i,j\in \mathbb{Z}_k}\gamma_{i,j}\widetilde{d}_{i,j}+\sum_{r,s\in\mathbb{Z}_k}\theta_{r,s}\widetilde{e}_{r,s},
	\]
	with 
	\begin{equation*}
	\gamma_{i,j}=\sum_{m,n\in \mathbb{Z}_k}\alpha_{m,n}\beta_{i-m,j-n}+\frac{1}{k}\sum_{r,s\in\mathbb{Z}_k}\eta^{i(r-s)}\kappa_{r,s}\omega_{r+j,s+j},~~i,j\in \mathbb{Z}_k,
	\end{equation*}
	\begin{equation*}
	\theta_{r,s}=\sum_{i,j\in \mathbb{Z}_k}\eta^{i(s-r)}(\alpha_{i,j}\omega_{r+j,s+j}+\beta_{i,j}\kappa_{r-j,s-j}),~~r,s\in\mathbb{Z}_k.
	\end{equation*}
\end{lem}	

This gives a characterization of $\text{Idem}(\mathcal{A}_k)$:

\begin{lem}[\cite{Sekine,FSonidempotentstatesonquantumgroups}]\label{characterization of idempotent states on Sekine quantum groups}
	$\mu=\sum_{i,j\in \mathbb{Z}_k}\alpha_{i,j}\widetilde{d}_{i,j}+\sum_{r,s\in\mathbb{Z}_k}\kappa_{r,s}\widetilde{e}_{r,s}\in\mathcal{A}'_k$ is an idempotent state if and only if $\alpha_{i,j}\geq 0$ for all $i,j\in\mathbb{Z}_k$, $K:=[\kappa_{r,s}]_{r,s\in\mathbb{Z}_k}$ is positive semi-definite, and the following equations hold:  
	\begin{equation}\label{Equation A}
	\alpha_{i,j}=\sum_{r,s\in \mathbb{Z}_k}\alpha_{i-r,j-s}\alpha_{r,s}+\frac{1}{k}\sum_{r,s\in\mathbb{Z}_k}\eta^{i(r-s)}\kappa_{r,s}\kappa_{r+j,s+j},~~i,j\in \mathbb{Z}_k,
	\end{equation}
	\begin{equation}\label{Equation B}
	\kappa_{r,s}=\sum_{i,j\in \mathbb{Z}_k}\eta^{i(s-r)}\alpha_{i,j}(\kappa_{r+j,s+j}+\kappa_{r-j,s-j}),~~r,s\in\mathbb{Z}_k,
	\end{equation}
	\begin{equation}\label{Equation C}
	\mu(1)=\sum_{i,j\in \mathbb{Z}_k}\alpha_{i,j}+\sum_{r\in\mathbb{Z}_k}\kappa_{r,r}=1.
	\end{equation}
\end{lem}

Certainly the Haar state on $\mathcal{A}_k$ 
\[
h_{\mathcal{A}_k}:=\frac{1}{2k^2}\sum_{i,j\in \mathbb{Z}_k}\widetilde{d}_{i,j}+\frac{1}{2k}\sum_{r\in\mathbb{Z}_k}\widetilde{e}_{r,r},
\]  
is a Haar idempotent. Franz and Skalski have determined all the Haar idempotents \cite[Theorem 6.5]{FSonidempotentstatesonquantumgroups} on $\mathcal{A}_k$. See also Proposition \ref{Haar idempotent states} below. They have also given some examples of non-Haar idempotents \cite[Proposition 6.6]{FSonidempotentstatesonquantumgroups}:
\begin{equation}\label{FS nonHaar example 1}
\phi_l=\frac{1}{2k}\sum_{i\in\mathbb{Z}_k}\widetilde{d}_{i,0}+\frac{1}{2}\widetilde{e}_{l,l},~~l\in\mathbb{Z}_k.
\end{equation}
There are certainly other non-Haar idempotents for special $k$'s, as pointed out at the end of \cite{FSonidempotentstatesonquantumgroups} with the following examples:
\begin{equation}\label{FS nonHaar example 2}
\frac{1}{4km}\sum_{i\in\mathbb{Z}_k}\sum_{l=0}^{m-1}\widetilde{d}_{i,lp}+\frac{1}{2m}\sum_{l=0}^{m-1}\widetilde{e}_{lp,lp},
\end{equation}
whenever $k=pm$ and $p,m\in\mathbb{N}$ such that $p,m\geq2$. 

With the help of some elementary number theoretic considerations, we solve the equations in Corollary \ref{characterization of idempotent states on Sekine quantum groups}. 
We will see in the following that the set of idempotent states, other than the Haar state, can be divided into three disjoint classes, denoted by $\mathcal{I}_1(\mathcal{A}_k),\mathcal{I}_2(\mathcal{A}_k)$ and $\mathcal{I}_3(\mathcal{A}_k)$. $\mathcal{I}_1(\mathcal{A}_k)$ consists of all Haar idempotents except the Haar state. $\mathcal{I}_2(\mathcal{A}_k)$ are non-Haar idempotents such that the corresponding matrix $K=[\kappa_{r,s}]_{r,s\in\mathbb{Z}_k}$ is diagonal, which include both \eqref{FS nonHaar example 1} and \eqref{FS nonHaar example 2} as subclasses. The third class $\mathcal{I}_3(\mathcal{A}_k)$, which are non-Haar idempotent states with $K$ not diagonal, is an unexpected new discovery.

Note first that if $\mu=\sum_{i,j\in \mathbb{Z}_k}\alpha_{i,j}\widetilde{d}_{i,j}+\sum_{r,s\in\mathbb{Z}_k}\kappa_{r,s}\widetilde{e}_{r,s}$ is an idempotent state, then from \eqref{Equation B} it follows that 
\begin{equation}\label{Equation B bis}
\sum_{r\in\mathbb{Z}_k}\kappa_{r,r}
=2\sum_{i,j\in\mathbb{Z}_k}\alpha_{i,j}\cdot\sum_{r\in\mathbb{Z}_k}\kappa_{r,r}.
\end{equation}
Together with \eqref{Equation C}, we have either 
\begin{equation}\label{Case 1:Tr(K)=0}
\sum_{i,j\in \mathbb{Z}_k}\alpha_{i,j}=1,\sum_{r\in\mathbb{Z}_k}\kappa_{r,r}=0,
\end{equation}	  
or 
\begin{equation}\label{Case 2:Tr(K)=1/2}
\sum_{i,j\in \mathbb{Z}_k}\alpha_{i,j}=\sum_{r\in\mathbb{Z}_k}\kappa_{r,r}=\frac{1}{2}.
\end{equation}	  

The following proposition characterizes all the idempotents verifying \eqref{Case 1:Tr(K)=0}. By Theorem 6.5 in \cite{FSonidempotentstatesonquantumgroups}, such idempotent states, together with the Haar state $h_{\mathcal{A}_k}$, form the family of Haar idempotents.

\begin{prop}\label{Haar idempotent states}
	Let $\mu=\sum_{i,j\in \mathbb{Z}_k}\alpha_{i,j}\widetilde{d}_{i,j}+\sum_{r,s\in\mathbb{Z}_k}\kappa_{r,s}\widetilde{e}_{r,s}\in \mathcal{A}'_k$. Then it is an idempotent state verifying $\sum_{i,j\in \mathbb{Z}_k}\alpha_{i,j}=1$ and $\sum_{r\in\mathbb{Z}_k}\kappa_{r,r}=0$ if and only if 
	\begin{enumerate}
		\item $\kappa_{r,s}=0$ for all $r,s\in\mathbb{Z}_k$;
		\item $\Gamma:=\{(i,j)\in\mathbb{Z}_k\times\mathbb{Z}_k:\alpha_{i,j}\neq0\}$ is a subgroup of $\mathbb{Z}_k\times\mathbb{Z}_k$ and $\alpha_{i,j}=\frac{1}{\sharp\Gamma},(i,j)\in\Gamma$.
	\end{enumerate}
	Moreover, in this case, such an idempotent state is Haar idempotent. Conversely, any Haar idempotent is either equal to the Haar state $h_{\mathcal{A}_k}$ or of this form: 
	\[
	h_\Gamma:=\frac{1}{\sharp\Gamma}\sum_{(i,j)\in \Gamma}\widetilde{d}_{i,j},
	\]
	with $\Gamma$ a subgroup of $\mathbb{Z}_k\times\mathbb{Z}_k$.
\end{prop}	   

\begin{proof}
	Since $K=[\kappa_{r,s}]_{r,s\in\mathbb{Z}_k}\geq 0$ and $\text{Tr}(K)=\sum_{r\in\mathbb{Z}_k}\kappa_{r,r}=0$, we have $\kappa_{rs}=0$ for all $r,s\in \mathbb{Z}_k$. Then \eqref{Equation B} is trivial and \eqref{Equation A} becomes 
	\begin{equation}\label{Equation A bis}
	\alpha_{i,j}=\sum_{r,s\in \mathbb{Z}_k}\alpha_{i-r,j-s}\alpha_{r,s},~~i,j\in \mathbb{Z}_k.
	\end{equation}
	From $\sum_{i,j\in \mathbb{Z}_k}\alpha_{i,j}=1$ it follows that $\alpha_{i',j'}\neq0$ for some $(i',j')\in\mathbb{Z}_k\times\mathbb{Z}_k$. So $\Gamma\neq\emptyset$. From \eqref{Equation A bis} and the non-negativity of $\alpha_{i,j}$ we have that $(i_1,j_1)\in\Gamma$ and $(i_2,j_2)\in\Gamma$ imply $(i_1+i_2,j_1+j_2)\in\Gamma$. Thus $\Gamma$ is closed under group action. Moreover, $(0,0)=(ki',kj')\in\Gamma$, i.e., $\Gamma$ contains the unit. Set $M:=\max_{i,j\in\mathbb{Z}_k}\alpha_{i,j}$ and suppose that it is attained by $\alpha_{i_0,j_0}$ with $(i_0,j_0)\in\mathbb{Z}_k\times\mathbb{Z}_k$. Clearly $M>0$. From 
	\[
	M=\sum_{r,s\in \mathbb{Z}_k}\alpha_{i_0-r,j_0-s}\alpha_{r,s}\leq M\sum_{r,s\in \mathbb{Z}_k}\alpha_{r,s}=M,
	\] 
	it follows that $\alpha_{i_0-r,j_0-s}=M$ as long as $\alpha_{r,s}\neq0$. So $\sharp\Gamma\leq\sharp\{(i,j):\alpha_{i,j}=M\}\leq\sharp\Gamma$, that is to say, $\alpha_{i,j}=\frac{1}{\sharp\Gamma}$ for any $(i,j)\in\Gamma$. Now apply $(i_0,j_0)$ to $(0,0)$, we have $\alpha_{-r,-s}=M$ as long as $\alpha_{r,s}\neq0$, i.e., $(-r,-s)\in\Gamma$ if $(r,s)\in\Gamma$. So any $(r,s)\in\Gamma$ has its inverse $(-r,-s)$ in $\Gamma$. Hence $\Gamma$ is a subgroup of $\mathbb{Z}_k\times\mathbb{Z}_k$.
	
	The remaining is a direct consequence of Theorem 6.5 of \cite{FSonidempotentstatesonquantumgroups}.
\end{proof}	

\begin{rem}\label{remark of prop1}
	Let $\mu$ be as above. If $\mu=h_{\Gamma}$, we have 
	\begin{equation*}
	\hat{\mu}(\pi_{i,j})\in\left\{
	\begin{pmatrix*}
	1&0\\
	0&1
	\end{pmatrix*},
	\begin{pmatrix*}
	1&0\\
	0&0
	\end{pmatrix*},
	\begin{pmatrix*}
	0&0\\
	0&1
	\end{pmatrix*},
	\begin{pmatrix*}
	0&0\\
	0&0
	\end{pmatrix*}\right\},
	~~i,j\in\mathbb{Z}_k.
	\end{equation*}
\end{rem}

The following theorem contains the main result of this paper, characterizing the set of idempotent states verifying \eqref{Case 2:Tr(K)=1/2}, which consists of the Haar state $h_{\mathcal{A}_k}$ and all 
the non-Haar idempotents.

Before this we need the following well-known \emph{B\'ezout's identity}: 

\begin{lem}[B\'ezout's identity]\label{Bezout's identity}
	For any integers $a,b\geq1$, there exist integers $m,n$ such that $ma+nb=\gcd(a,b)$, where $\gcd(a,b)$ denotes the greatest common divisor of $a$ and $b$.	
\end{lem}

Note that we can choose $m>0,n<0$ or $m<0,n>0$ freely. Indeed, we can replace the pair $(m,n)$ with $(m+lb,n-la)$ for any $l\in\mathbb{Z}$.

\begin{thm}\label{thm:non-Haar}
	Let $\mu=\sum_{i,j\in \mathbb{Z}_k}\alpha_{i,j}\widetilde{d}_{i,j}+\sum_{r,s\in\mathbb{Z}_k}\kappa_{r,s}\widetilde{e}_{r,s}\in \mathcal{A}'_k$. Then it is an idempotent state verifying $\sum_{i,j\in \mathbb{Z}_k}\alpha_{i,j}=\sum_{r\in\mathbb{Z}_k}\kappa_{r,r}=\frac{1}{2}$ if and only if $\mu$ is either:
	\begin{enumerate}
		\item [(1)] the Haar state $h_{\mathcal{A}_k}:=\frac{1}{2k^2}\sum_{i,j\in \mathbb{Z}_k}\widetilde{d}_{i,j}+\frac{1}{2k}\sum_{r\in\mathbb{Z}_k}\widetilde{e}_{r,r}$; or
		\item [(2)]
		\begin{equation*}
		h_{\Gamma,l}:=\frac{1}{2\sharp\Gamma}\sum_{(i,j)\in\Gamma}\widetilde{d}_{i,j}+\frac{q}{2k}\sum_{r\equiv l\mod q}\widetilde{e}_{r,r},
		\end{equation*}
		where $\Gamma=\mathbb{Z}_k\times q\mathbb{Z}_k$ with $q|k$, $q>1$, and $l\in\mathbb{Z}_q$; or
		\item [(3)]
		\begin{equation*}
		h_{\Gamma,l,\tau}:=\frac{1}{2\sharp\Gamma}\sum_{(i,j)\in\Gamma}\widetilde{d}_{i,j}+\frac{q}{2k}\sum_{r,s\equiv l\mod q}\tau_{s-r}\widetilde{e}_{r,s},    	  	  \end{equation*}
		where $\Gamma=p\mathbb{Z}_k\times q\mathbb{Z}_k$ with $p>1$ and $pq=k$, $l\in\mathbb{Z}_q$, and $\tau=(\tau_j)_{j\in q\mathbb{Z}_k}\in\{\pm1\}^{q\mathbb{Z}_k}$ such that 
		\begin{equation}\label{a formula for K to be positive}
		\sum_{j\in q\mathbb{Z}_k}\tau_j\eta^{ij}\geq0,~~i\in\mathbb{Z}_{k/q}.
		\end{equation}
	\end{enumerate} 
\end{thm}

\begin{rem}\label{remark of tau}
	\eqref{a formula for K to be positive} is equivalent to positive semi-definiteness of $[\kappa_{r,s}]_{r,s\in\mathbb{Z}_k}$. Such $\tau$ always exists as one can choose $\tau_j=1$ for all $j$, which is kind of trivial. It is not difficult to construct non-trivial ones. For example, when $k=2$, there is an another $\tau'$ with $\tau'_0=1$ and $\tau'_1=-1$ which satisfies \eqref{a formula for K to be positive}.
\end{rem}

\begin{proof}[Proof of Theorem \ref{thm:non-Haar}]
	Observe first that for any $j\in \mathbb{Z}_k$, $K=[\kappa_{r,s}]_{r,s\in \mathbb{Z}_k}\geq 0$ implies $[\kappa_{r+j,s+j}]_{r,s\in \mathbb{Z}_k}\geq 0$, thus their Hadamard product $[\kappa_{rs}\kappa_{r+j,s+j}]_{r,s\in \mathbb{Z}_k}\geq 0$. So we have for any $i,j\in \mathbb{Z}_k$
	\begin{equation}\label{Hadamard product}
	\sum_{r,s\in \mathbb{Z}_k}\eta^{i(r-s)}\kappa_{r,s}\kappa_{r+j,s+j}
	=\sum_{r,s\in \mathbb{Z}_k}\eta^{ir}\overline{\eta^{is}}\kappa_{r,s}\kappa_{r+j,s+j}
	\geq 0.
	\end{equation}
	Recall that $\Gamma=\{(i,j)\in\mathbb{Z}_k\times\mathbb{Z}_k:\alpha_{i,j}\ne0\}$. Since $\sum_{i,j\in\mathbb{Z}_k}\alpha_{i,j}=\frac{1}{2}$, there exists $(i',j')\in\mathbb{Z}_k\times\mathbb{Z}_k$ such that $\alpha_{i',j'}\neq 0$. This allows us to define 
	\[
	p:=\min \{i>0:(i,j)\in\Gamma\text{ for some } j\in \mathbb{Z}_k\},
	\]
	\[
	q:=\min \{j>0:(i,j)\in\Gamma\text{ for some } i\in \mathbb{Z}_k\}.
	\]	
	
	\textbf{Claim 1}: \emph{We have $p|k,q|k$, and for any $(i,j)\in\mathbb{Z}_k\times\mathbb{Z}_k$}
	
	\begin{equation}\label{claim 1}
	\alpha_{i,j}\neq 0 \Rightarrow p|i \text{ and } q|j.
	\end{equation}
	
	To show this, recall that $\alpha_{i,j}\geq 0$, so $\alpha_{i,j}\neq 0$ simply means $\alpha_{i,j}> 0$. From \eqref{Equation A} and the previous observation \eqref{Hadamard product}, it follows that 
	\begin{equation}\label{Obeservation 1}
	(i_1,j_1),(i_2,j_2)\in\Gamma\Rightarrow (m_1 i_1+m_2 i_2,m_1 j_1+m_2 j_2)\in\Gamma,~~m_1,m_2\in\mathbb{Z}_{\geq 0},
	\end{equation}
	and
	\begin{equation}\label{Obeservation 1 bis}
	(i_1,j_1)\notin\Gamma,(i_2,j_2)\in\Gamma \Rightarrow (i_1-m i_2, j_1-m j_2)\notin\Gamma,~~m\in\mathbb{Z}_{\geq 1},
	\end{equation}
	From Lemma \ref{Bezout's identity}, there exist integers $m,n>0$ such that $mp-nk=\gcd(p,k)\leq p$. Suppose that $(p,j_p)\in\Gamma$ for some $j_p>0$, then \eqref{Obeservation 1} and Lemma \ref{Bezout's identity} yield 
	\[
	(mp,mj_p)=(mp-nk,mj_p)=\left(\gcd(p,k),mj_p\right)\in\Gamma.
	\] 
	From the definition of $p$ we have $\gcd(p,k)=p$, i.e. $p|k$. 
	
	For any $\alpha_{i,j}\neq 0$, i.e. $(i,j)\in\Gamma$, there exist, by applying Lemma \ref{Bezout's identity} two times, integers $m,n,l>0$ such that $0<mp+ni-lk=\gcd(p,i,k)\leq p$. Thus \eqref{Obeservation 1} and Lemma \ref{Bezout's identity} yield
	\[
	(mp+ni,mj_p+nj)=(mp+ni-lk,mj_p+nj)=(\gcd(p,i,k),mj_p+nj)\in\Gamma.
	\]
	So $\gcd(p,i,k)=p$, i.e. $p|i$, which finishes the proof of \textbf{Claim 1} for $p$. The proof for $q$ is similar. \\
	
	\textbf{Claim 2}: \emph{Fix $t\in\mathbb{Z}_k$. Suppose that $|\kappa_{r_0,s_0}|=\max\{|\kappa_{r,s}|:(r,s)\in\mathbb{Z}_k\times\mathbb{Z}_k,r-s=t\}$. Then for any $(i,j)\in\Gamma$, we have
		\begin{equation}\label{claim 2}
		\kappa_{r_0,s_0}=\eta^{-it}\kappa_{r_0+j,s_0+j}=\eta^{-it}\kappa_{r_0-j,s_0-j}.
		\end{equation}
		Moreover, $\eta^{2pt}=1$.}\\
	
	In fact, from \eqref{Equation B} it follows 
	\[
	|\kappa_{r_0,s_0}|
	\leq\sum_{i,j\in \mathbb{Z}_k}\alpha_{i,j}|\eta^{-it}(\kappa_{r_0+j,s_0+j}+\kappa_{r_0-j,s_0-j})|
	\leq2|\kappa_{r_0,s_0}|\sum_{i,j\in \mathbb{Z}_k}\alpha_{i,j}
	=|\kappa_{r_0,s_0}|.
	\]
	Thus $\alpha_{i,j}\neq 0$ implies \eqref{claim 2}. Consequently, we have $|\kappa_{r_0+j, s_0+j}|=|\kappa_{r_0-j, s_0-j}|=|\kappa_{r_0,s_0}|$. Repeating this argument for $\kappa_{r_0+j,s_0+j}$ and $\kappa_{r_0-j,s_0-j}$, we have finally 
	\begin{equation}\label{relation1 of kappa_rs:same module}
	|\kappa_{r,s}|=|\kappa_{r_0,s_0}| \text{ for any $r,s$ such that }j|r-r_0=s-s_0,
	\end{equation}
	whenever $\alpha_{i,j}\neq0$ for some $i$. Moreover, \eqref{claim 2} implies $\kappa_{r_0, s_0}=\eta^{-it}\kappa_{r_0+j,s_0+j}=\eta^{-2it}\kappa_{r_0,s_0}$, so we have $\kappa_{r_0,s_0}\neq 0$ only if $\eta^{2it}=1$, which, by the definition of $p$, yields $\eta^{2pt}=1$. So \textbf{Claim 2} is proved.\\
	
	Recall that $\kappa_{r,r}\geq 0$ for all $r\in\mathbb{Z}_k$, since $K\geq 0$. From $\sum_{r\in\mathbb{Z}_k}\kappa_{r,r}=\frac{1}{2}$, we have $\kappa_{l,l}=\max_{r\in \mathbb{Z}_k}\kappa_{r,r}> 0$ for some $l$. Suppose $\alpha_{i_q,q}\neq 0$, then \eqref{relation1 of kappa_rs:same module} implies 
	\begin{equation}\label{kappa_rr nonzero}
	\kappa_{r,r}=\kappa_{l,l}>0, r\equiv l \mod{q}.
	\end{equation}
	For convenience, let $0\leq l <q$. From \eqref{Equation A} and \eqref{claim 1}, we have for any $i\in\mathbb{Z}_k$ and any $q\nmid j$
	\begin{equation*}
	\sum_{r,s\in \mathbb{Z}_k}\eta^{i(r-s)}\kappa_{r,s}\kappa_{r+j,s+j}=0.
	\end{equation*}
	So for any $q\nmid j$
	\begin{align*}
	0&=\sum_{i\in\mathbb{Z}_k}\sum_{r,s\in \mathbb{Z}_k}\eta^{i(r-s)}\kappa_{r,s}\kappa_{r+j,s+j}\\
	&=\sum_{r,s\in\mathbb{Z}_k}\kappa_{r,s}\kappa_{r+j,s+j}\sum_{i\in \mathbb{Z}_k}\eta^{i(r-s)}\\
	&=k\sum_{r\in \mathbb{Z}_k}\kappa_{r,r}\kappa_{r+j,r+j}.
	\end{align*}
	
	Thus $\kappa_{r,r}\kappa_{r+j,r+j}=0$ whenever $r\in\mathbb{Z}_k$ and $q\nmid j$. Combining this with \eqref{kappa_rr nonzero} we obtain 
	\begin{equation}\label{relation of kappa_rr}
	\kappa_{r,r}
	=\begin{cases}
	\frac{q}{2k} & r\equiv l \mod{q}\\
	0 & \text{otherwise }
	\end{cases}.
	\end{equation}
	From this and the positive semi-definitiveness of $K$ we have 
	\begin{equation}\label{relation2 of kappa_rs: zero if}
	\kappa_{r,s}=0, \text{ if either } q\nmid r-l \text{ or } q\nmid s-l.
	\end{equation} 
	So it remains to compute the submatrix $[\kappa_{r,s}]_{r,s\equiv l\mod{q}}$. For this set 
	\begin{equation*}
	p':=\min\{i>0:(i,0)\in\Gamma\},
	\end{equation*}
	\begin{equation*}
	q':=\min\{j>0:(0,j)\in\Gamma\}.
	\end{equation*}
	These are well-defined, since $(k,0)=(0,k)=(0,0)\in\Gamma$. Indeed, suppose $(i',j')\in\Gamma$, then from \eqref{Obeservation 1} it follows $(0,0)=(ki',kj')\in\Gamma$.
	
	We have also by \eqref{Obeservation 1}
	\begin{equation}\label{property of p' and q'}
	p'|i\text{ and }q'|j\Rightarrow(i,j)\in\Gamma.
	\end{equation} 
	So from \eqref{claim 1} it follows $p|p'$ and $q|q'$.\\
	
	\textbf{Claim 3:} $p=p'$ and $q=q'$. As a consequence, we have 
	\begin{equation}\label{alpha_i,j nonzero iff}
	(i,j)\in\Gamma \text{ if and only if } p|i \text{ and } q|j.
	\end{equation}
	
	To prove this, note first that for any $p|i$, there exists $j\in\mathbb{Z}_k$ such that $(i,j)\in\Gamma$. Otherwise, $(i,j)\notin\Gamma$ for all $j\in\mathbb{Z}_k$. Since $(p,j_p)\notin\Gamma$ for some $j_p\in\mathbb{Z}_k$, we have by \eqref{Obeservation 1 bis} that $(i-p,j)\notin\Gamma$ for all $j\in\mathbb{Z}_k$. This argument gives finally $(0,j)\notin\Gamma$ for all $j\in\mathbb{Z}_k$, which contradicts with the fact that $(0,0)\in\Gamma$. Similarly, for any $q|j$, there exists $i\in\mathbb{Z}_k$ such that $(i,j)\in\Gamma$. This allows us to define for all $p|i$ and $q|j$
	\[
	p_j:=\min \{i>0:(i,j)\in\Gamma\},
	\]
	\[
	q_i:=\min \{j>0:(i,j)\in\Gamma\}.
	\]
	Then \eqref{claim 1} implies $p|p_j$ and  $q|q_i$ for all such $i,j$. Following a similar argument of showing \eqref{claim 1}, we have by \eqref{Obeservation 1} and Lemma \ref{Bezout's identity} that $p_j|p'$ and $q_i|q'$. Moreover, note that 
	\[
	(0,j)\in\Gamma \text{ iff } q'|j,
	\text{ and }
	(i,0)\in\Gamma \text{ iff } p'|i,
	\] 
	which is a consequence of definitions of $p',q'$ and \eqref{Obeservation 1} together with B\'ezout's identity. Then \eqref{Obeservation 1} tells us that for any $p|i$
	\begin{equation}
	(i,q_i+mq')\in\Gamma,~~m\geq 0;
	\end{equation}
	and \eqref{Obeservation 1 bis} tells us that for any $p|i$
	\begin{equation}
	(i,j-mq_i)=(0-(-i),j-mq_i)\notin\Gamma,~~q'\nmid j,m\geq 0.
	\end{equation}
	So we have for any $p|i$, 
	\[
	(i,j)\in\Gamma \text{ iff } j\equiv q_i\mod q'. 
	\]   
	Now we are ready to prove \textbf{Claim 3}. Suppose $q\neq q'$, then for any $p|i$, we have $(i,j(i))\notin\Gamma$ for some $q|j(i)$ (for example, take $j(i)=q_i+q$). So $\alpha_{i,j(i)}=0$, and thus from the non-negativity of $\alpha_{m,n}$ and \eqref{Hadamard product} we have 
	\begin{equation*}
	\sum_{r,s\in \mathbb{Z}_k}\eta^{i(r-s)}\kappa_{r,s}\kappa_{r+j(i),s+j(i)}=0.
	\end{equation*}
	By \eqref{relation2 of kappa_rs: zero if}, it becomes
	\begin{equation*}
	\sum_{r,s\in \mathbb{Z}_k}\eta^{i(r-s)}\kappa_{r,s}\kappa_{r+j(i),s+j(i)}
	=\sum_{r,s\equiv l \mod q}\eta^{i(r-s)}\kappa_{r,s}\kappa_{r+j(i),s+j(i)}
	=0.
	\end{equation*}
	Fix $p|i$ and the associated $j(i)$, then we have for any $q|t$ 
	\begin{equation}\label{Observation: independent of r,s}
	\kappa_{r,s}\kappa_{r+j(i),s+j(i)}=d_i \kappa_{t+l,l}^2,\text{ for all } r,s\equiv l \mod q \text{ and } r-s=t,
	\end{equation}
	where $d_i\in\{\pm1\}$ is independent of $r,s$. In fact, this is trivial when $\max_{r,s:r-s=t}|\kappa_{rs}|=0$. Set $d_i\equiv 1$ for example. If $\max_{r,s:r-s=t}|\kappa_{rs}|>0$, recall that we have for such $r,s$
	\[
	\kappa_{r,s}=\eta^{p_q(s-r)}\kappa_{r-q,s-q}=\eta^{-p_q t}\kappa_{r-q,s-q}~~\text{and}~~\eta^{2p_q (r-s)}=\eta^{2p_q t}=1.
	\]
	Clearly, $\eta^{p_q t}=\pm1$. So \eqref{Observation: independent of r,s} is also trivial if $\eta^{p_q t}=1$: we can choose $d_i=1$. If $\eta^{p_q t}=-1$, then for $r,s\equiv l \mod q$
	\[
	\kappa_{r,s}\kappa_{r+j(i),s+j(i)}=(-1)^{j(i)/q}\kappa^2_{r-s+l,l},
	\]
	and we can choose $d_i=(-1)^{j(i)/q}$. Thus \eqref{Observation: independent of r,s} holds.\\
	
	Now we have for any $p|i$
	\begin{align*}
	0&=\sum_{r,s\equiv l \mod q}\eta^{i(r-s)}\kappa_{r,s}\kappa_{r+j(i),s+j(i)}\\
	&=\sum_{r,s=tq+l,t\in\mathbb{Z}_{k/q}}\eta^{i(r-s)}d_i \kappa^2_{r-s+l,l}\\
	&=\frac{k}{q}\sum_{t\in \mathbb{Z}_{k/q}}\eta^{itq}d_i \kappa^2_{tq+l,l}.
	\end{align*}
	This clearly holds for any $p\nmid i$ and any $d_i$ because of \eqref{claim 1} and \eqref{relation2 of kappa_rs: zero if}. Set $d_i:=1$ for all $p\nmid i$. Thus we have for each $i\in\mathbb{Z}_{k}$ 
	\begin{equation*}
	\sum_{t\in \mathbb{Z}_{k/q}}d_i(\eta^{q})^{it} \kappa^2_{tq+l,l}
	=0.
	\end{equation*}
	Then the system of linear equations
	\begin{equation*}
	\sum_{t\in \mathbb{Z}_{k/q}}d_i(\eta^{q})^{it} \kappa^2_{tq+l,l}
	=0,~~ i=0,1,\dots,\frac{k}{q}-1.
	\end{equation*}
	can be represented as 
	\begin{equation}\label{Matrix equation}
	DVX=0,
	\end{equation}
	where $D=\text{diag}(d_0,d_1\dots,d_{\frac{k}{q}-1})$ is an invertible $\frac{k}{q}\times\frac{k}{q}$ diagonal matrix, $V=V(1,\eta^q,\dots,\eta^{k-q})$ is an $\frac{k}{q}\times\frac{k}{q}$ Vandermonde matrix, and $X=(\kappa^2_{l,l},\kappa^2_{q+l,l}\dots,\kappa^2_{k-q+l,l})^{T}$ is a $\frac{k}{q}$ dimensional vector. Here $V(a_1,a_2,\dots,a_n)$ denotes the $n\times n$ Vandermonde matrix $[a_i^{j-1}]_{i,j=1}^{n}$. By the definition of $\eta$, $V$ is invertible, so we have $X=0$. But $\kappa^2_{l,l}=\frac{q^2}{4k^2}\neq 0$, which leads to a contradiction! Hence $q=q'$.
	
	Now $p=p'$ follows directly. Indeed, $q=q'$ implies $q_p=q$, so we have $(p,q)=(p,q_p)\in\Gamma$. Thus $(p,0)=(p,q)+(0,-q)\in\Gamma$. Hence $p=p'$, which ends the proof of \textbf{Claim 3}.\\
	
	\textbf{Claim 4:} We have either $p=1$ or $pq=k$ for $p\geq 2$.\\
	
	We will use a similar argument as above. Before this, let us update several conclusions, following \textbf{Claim 3}. Note first that $\kappa_{r,s}=\kappa_{r-s+l,l}$ for all $r,s\equiv l \mod q$, since $(0,q)\in\Gamma$. Moreover, $\kappa_{r,s}\neq 0$ only if $\eta^{p(r-s)}=1$.
	
	If $p>1$, we have for any $p\nmid i$:
	\begin{align*}
	0&=\sum_{r,s\in\mathbb{Z}_k}\eta^{i(r-s)}\kappa_{r,s}\kappa_{r+j,s+j}\\
	&=\sum_{r,s\equiv l \mod q}\eta^{i(r-s)}\kappa_{r,s}\kappa_{r+j,s+j}\\
	&=\sum_{r,s=tq+l,t\in\mathbb{Z}_{k/q}}\eta^{i(r-s)}\kappa^2_{r-s+l,l}\\
	&=\frac{k}{q}\sum_{t\in \mathbb{Z}_{k/q}}\eta^{itq}\kappa^2_{tq+l,l}.
	\end{align*}
	So we have the following system of linear equations:
	\begin{equation*}
	\sum_{t\in \mathbb{Z}_{k/q}}(\eta^{q})^{it} \kappa^2_{tq+l,l}
	=0,~~i\in\{0,1,2,\dots,\frac{k}{q}-1\}\setminus p\mathbb{Z}_k,
	\end{equation*}
	which can be represented as
	\begin{equation*}
	V'X'=0,
	\end{equation*}
	where $V'$ is a $m\times \frac{k}{q}$ submatrix of $V$, the Vandermonde matrix introduced earlier, with $m=\frac{k}{q}-[\frac{k}{pq}]$ and $X'=X=(\kappa^2_{l,l},\kappa^2_{q+l,l}\dots,\kappa^2_{k-q+l,l})^{T}$ with $n$ nonzero entries. Since $\kappa^2_{r-s+l,l}\neq 0$ only if $q|r-s$ and $\frac{k}{p}|r-s$ (because $\eta^{p(r-s)}=1$), we have $n\leq \frac{k}{\text{lcm}(\frac{k}{p},q)}=\frac{p}{q}\gcd(\frac{k}{p},q)$, where lcm$(a,b)$ denotes the least common multiple of $a$ and $b$. The fact that $\kappa^2_{l,l}=\frac{q^2}{4k^2}\neq 0$ requires $m<n$.\\
	
	If $\gcd(\frac{k}{p},q)<\frac{k}{p}$, then $\gcd(\frac{k}{p},q)\leq\frac{k}{2p}$ and for $p\geq 2$,
	\begin{equation*}
	m=\frac{k}{q}-[\frac{k}{pq}]
	\geq\frac{k}{q}-\frac{k}{pq}
	\geq \frac{k}{2q}
	=\frac{p}{q}\cdot \frac{k}{2p}
	\geq \frac{p}{q}\gcd(\frac{k}{p},q)
	\geq n,
	\end{equation*}
	which leads to a contradiction. So $\gcd(\frac{k}{p},q)=\frac{k}{p}$, i.e., $k|pq$. Thus $\frac{k}{q}\geq n>m=\frac{k}{q}-[\frac{k}{pq}]\geq \frac{k}{q}-1$. This happens only if $k=pq$, which ends the proof of the \textbf{Claim 4}.\\
	
	Now we are ready to finish the proof of the theorem.
	
	(i) Suppose $p=1$. In this case $\Gamma=\{(i,j)\in\mathbb{Z}_k\times\mathbb{Z}_k:q|j\}$, and  
	\begin{equation*}
	\kappa_{rs}= \begin{cases}
	\frac{1}{2p} & r=s\equiv l \mod{q}\\
	0 & \text{otherwise}.\
	\end{cases}
	\end{equation*}
	Hence \eqref{Equation A} becomes
	\begin{equation*}
	\alpha_{i,j}=\sum_{(r,s)\in\Gamma}\alpha_{i-r,j-s}\alpha_{r,s}+\frac{q}{4k^2},~~ q|j.
	\end{equation*}
	Let $M:=\max\{\alpha_{i,j}:(i,j)\in\Gamma\}$ and suppose that $\alpha_{i_0,j_0}=M$. Then 
	\begin{equation*}
	M=\alpha_{i_0,j_0}=\sum_{(r,s)\in\Gamma}\alpha_{i_0-r,j_0-s}\alpha_{r,s}+\frac{q}{4k^2}\leq M\sum_{(r,s)\in\Gamma}\alpha_{r,s}+\frac{q}{4k^2}=\frac{M}{2}+\frac{q}{4k^2},
	\end{equation*}
	which implies $M\leq\frac{q}{2k^2}$. Moreover, $M\geq\frac{1}{2\sharp\Gamma}=\frac{q}{2k^2}$. So $M=\frac{q}{2k^2}$, and thus 
	\begin{equation*}
	\alpha_{i,j}=\frac{q}{2k^2},~~ (i,j)\in\Gamma.
	\end{equation*}
	If $q=1$, we have $\mu=h_{\mathcal{A}_k}$, which is nothing but (1). Otherwise, we obtain (2).\\
	
	(ii) Suppose $p>1$ and $k=pq$. In this case $n=\frac{k}{q}=p,m=\frac{k}{q}-[\frac{k}{pq}]=p-1$. Since $\kappa^2_{l,l}=\frac{1}{4p^2}$, the equation $V'X'=0$ possesses exactly one solution: $X'=(\frac{1}{4p^2},\dots,\frac{1}{4p^2})^T$. That is to say,
	\begin{equation*}
	\kappa_{rs}= \begin{cases}
	\kappa_{0,0}=\frac{1}{2p} & r=s,r,s\equiv l \mod{q}\\
	\kappa_{0,s-r}=\pm\frac{1}{2p} & r\neq s,r,s\equiv l \mod{q}\\
	0 & \text{otherwise}.
	\end{cases}
	\end{equation*}
	So the fact that $K\geq 0$ is equivalent to the positive semi-definiteness of the circulant matrix $[\kappa_{r,s}]_{r,s\equiv l\mod{q}}$. Since the set of all the eigenvalues of $[\kappa_{r,s}]_{r,s\equiv l\mod{q}}$ is
	\[
	\{\sum_{j\in q\mathbb{Z}_k}\kappa_{0,j}\eta^{ij},i\in\mathbb{Z}_{k/q}\},
	\]
	we have $K\geq 0$ if and only if 
	\begin{equation*}
	\sum_{j\in q\mathbb{Z}_k}\tau_j\eta^{ij}\geq0,~~ i\in\mathbb{Z}_{k/q},
	\end{equation*}
	where $\tau_j:=\kappa_{0,j}\in\{\pm1\},j\in q\mathbb{Z}_k$.\\
	
	Now \eqref{Equation A} is equivalent to 
	\begin{equation*}
	\alpha_{i,j}=\sum_{(m,n)\in\Gamma}\alpha_{i-m,j-n}\alpha_{m,n}+\frac{1}{4k},~~ (i,j)\in\Gamma.
	\end{equation*} 
	Following a similar argument as above, we have 
	\begin{equation*}
	\alpha_{i,j}=\frac{1}{2k},~~ (i,j)\in\Gamma.
	\end{equation*} 
	This gives (3), and the proof of the theorem is complete.	
\end{proof}

\begin{rem}\label{remark of prop2}
	Let $\mu$ be as above. Then 
	\begin{enumerate}
		\item $\mu=h_{\Gamma,l}$ if and only if
		\begin{equation*}
		\mu(\rho_{i,j})=
		\begin{cases*}
		\frac{1}{2}& $i=0,\frac{k}{q}|j$\\
		0& \text{otherwise}
		\end{cases*}
		\text{ and }
		\mu(\sigma_{i,j})=
		\begin{cases*}
		\frac{1}{2}\eta^{jl}& $i=0,\frac{k}{q}|j$\\
		0& \text{otherwise}
		\end{cases*},
		\end{equation*}
		if and only if 
		\begin{equation*}
		\hat{\mu}(\pi_{i,j})=
		\begin{cases*}
		\begin{psmallmatrix*}
		\frac{1}{2}&\frac{1}{2}\eta^{-jl}\\
		\frac{1}{2}\eta^{jl}&\frac{1}{2}
		\end{psmallmatrix*}
		& $i=0,\frac{k}{q}|j$\\
		\begin{psmallmatrix*}
		0&0\\
		0&0
		\end{psmallmatrix*}
		& \text{otherwise}
		\end{cases*}
		\end{equation*}
		\item $\mu=h_{\Gamma,l,\tau}$ if and only if
		\begin{equation*}
		\mu(\rho_{i,j})=
		\begin{cases*}
		\frac{1}{2}& $q|i,p|j$\\
		0& \text{otherwise}
		\end{cases*}
		\text{ and }
		\mu(\sigma_{i,j})=
		\begin{cases*}
		\frac{1}{2}\tau_j\eta^{jl}& $q|i,p|j$\\
		0& \text{otherwise}
		\end{cases*},
		\end{equation*}
		if and only if 
		\begin{equation*}
		\hat{\mu}(\pi_{i,j})=
		\begin{cases*}
		\begin{psmallmatrix*}
		\frac{1}{2}&\frac{1}{2}\tau_{-j}\eta^{-jl}\\
		\frac{1}{2}\tau_j\eta^{jl}&\frac{1}{2}
		\end{psmallmatrix*}
		& $q|i,p|j$\\
		\begin{psmallmatrix*}
		0&0\\
		0&0
		\end{psmallmatrix*}
		& \text{otherwise}
		\end{cases*}
		\end{equation*}
	\end{enumerate}
\end{rem}    
Denote by $\mathcal{I}_1(\mathcal{A}_k)$, $\mathcal{I}_2(\mathcal{A}_k)$ and $\mathcal{I}_3(\mathcal{A}_k)$ the family of idempotent states of the forms $h_{\Gamma}$, $h_{\Gamma,l}$ and $h_{\Gamma,l,\tau}$ respectively. Then the discussions above can be rephrased as    
\begin{thm}
	Fix $k\geq2$ an integer. Then the family of idempotent states $\text{Idem}(\mathcal{A}_k)$ on Sekine quantum group $\mathcal{A}_k$ is given through 
	\[
	\text{Idem}(\mathcal{A}_k)=\{h_{\mathcal{A}_k}\}\cup\mathcal{I}_1(\mathcal{A}_k)\cup\mathcal{I}_2(\mathcal{A}_k)\cup\mathcal{I}_3(\mathcal{A}_k).
	\]
\end{thm}

\section{The order structure on $\text{Idem}(\mathcal{A}_k)$}
Franz and Skalski introduced in \cite{FSonidempotentstatesonquantumgroups} the order relation on the set of idempotent states of a finite quantum group. We recall this definition for Sekine quantum groups here.

\begin{defn}
	Let $\phi_1,\phi_2\in \text{Idem}(\mathcal{A}_k)$. Denote by $\prec$ the partial order on $\text{Idem}(\mathcal{A}_k)$ given through 
	\begin{equation*}
	\phi_1\prec \phi_2 \ \ \text{ if }\ \ \phi_1\star\phi_2=\phi_2.
	\end{equation*}
\end{defn}
In this order the Haar state $h_{\mathcal{A}_k}$ and the counit $\epsilon$ are, respectively, the biggest and smallest idempotent on $\text{Idem}(\mathcal{A}_k)$. Let $\mu,\nu\in\text{Idem}(\mathcal{A}_k)$. We use superscripts to label all the symbols which appeared before. For example, 
\begin{equation*}
\mu:=\sum_{i,j\in \mathbb{Z}_k}\alpha^{(\mu)}_{i,j}\widetilde{d}_{i,j}+\sum_{r,s\in\mathbb{Z}_k}\kappa^{(\mu)}_{r,s}\widetilde{e}_{r,s},~~ \nu:=\sum_{i,j\in \mathbb{Z}_k}\alpha^{(\nu)}_{i,j}\widetilde{d}_{i,j}+\sum_{r,s\in\mathbb{Z}_k}\kappa^{(\nu)}_{r,s}\widetilde{e}_{r,s}.
\end{equation*}
We introduce the partial order $\prec$ in the family of 2 by 2 idempotent matrices $\mathcal{J}:=\{A\in M_2(\mathbb{C}):A^2=A\}$:
\begin{equation*}
A\prec B \text{ if } AB=B,~~ A,B\in\mathcal{J},
\end{equation*}
and the partial order $\prec$ in the family of subgroups of $\mathbb{Z}_k\times\mathbb{Z}_k$:
\begin{equation*}
\Gamma\prec\Lambda \text{ if } \Gamma\subset\Lambda,\ \Lambda,~~ \Gamma\leq\mathbb{Z}_k\times\mathbb{Z}_k.
\end{equation*}
Then from \eqref{Fourier transform}, $\mu\prec\nu$ if and only if 
\begin{equation}\label{Observation of order}
\hat{\mu}(\pi_{i,j})\prec \hat{\nu}(\pi_{i,j}),~~ i,j\in\mathbb{Z}_k,
\end{equation}

Our main result in this section is the following theorem, characterizing the order structure in the lattice $(\text{Idem}(\mathcal{A}_k),\prec)$.

\begin{thm}\label{Order structure}
	Let $\mu,\nu$ be idempotent states, other than the Haar state, on $\mathcal{A}_k$ as above. Then $\mu\prec\nu$ if and only if one of the following holds:
	\begin{enumerate}
		\item [(1)] $\mu=h_{\Gamma^{(\mu)}},\nu=h_{\Gamma^{(\nu)}}\in\mathcal{I}_1(\mathcal{A}_k)$ and $\Gamma^{(\mu)}\prec\Gamma^{(\nu)}$;
		\item [(2)] $\mu=h_{\Gamma^{(\mu)}}\in\mathcal{I}_1(\mathcal{A}_k),\nu=h_{\Gamma^{(\nu)},l^{(\nu)}}\in\mathcal{I}_2(\mathcal{A}_k)$ and $\Gamma^{(\mu)}\prec\Gamma^{(\nu)}$;
		\item [(3)] $\mu=h_{\Gamma^{(\mu)}}\in\mathcal{I}_1(\mathcal{A}_k),\nu=h_{\Gamma^{(\nu)},l^{(\nu)},\tau^{(\nu)}}\in\mathcal{I}_3(\mathcal{A}_k)$ and $\Gamma^{(\mu)}\prec\Gamma^{(\nu)}$;
		\item [(4)] $\mu=h_{\Gamma^{(\mu)},l^{(\mu)}},\nu=h_{\Gamma^{(\nu)},l^{(\nu)}}\in\mathcal{I}_2(\mathcal{A}_k)$, $\Gamma^{(\mu)}\prec\Gamma^{(\nu)}$ and $l^{(\mu)}\equiv l^{(\nu)}\mod q^{(\nu)}$;
		\item [(5)] $\mu=h_{\Gamma^{(\mu)},l^{(\mu)},\tau^{(\mu)}}\in\mathcal{I}_3(\mathcal{A}_k),\nu=h_{\Gamma^{(\nu)},l^{(\nu)}}\in\mathcal{I}_2(\mathcal{A}_k)$, $\Gamma^{(\mu)}\prec\Gamma^{(\nu)}$ and 
		\begin{equation*}
		\tau^{(\nu)}_j=\eta^{j(l^{(\nu)}-l^{(\mu)})}\text{ for $j$ such that }\frac{k}{q^{(\nu)}}=p^{(\nu)}|j;
		\end{equation*}
		\item [(6)] $\mu=h_{\Gamma^{(\mu)},l^{(\mu)},\tau^{(\mu)}},\nu=h_{\Gamma^{(\nu)},l^{(\nu)},\tau^{(\nu)}}\in\mathcal{I}_3(\mathcal{A}_k)$, $\Gamma^{(\mu)}=\Gamma^{(\nu)}$ and 
		\begin{equation*}
		\tau^{(\nu)}_j=\tau^{(\mu)}_j\eta^{j(l^{(\nu)}-l^{(\mu)})}\text{ for $j$ such that } p^{(\mu)}=p^{(\nu)}|j.
		\end{equation*}
	\end{enumerate}
\end{thm}

\begin{proof}
	By \eqref{Observation of order}, if $\hat{\mu}(\pi_{i,j})\prec \hat{\nu}(\pi_{i,j})$, then $\hat{\mu}(\pi_{i,j})=0$ implies $\hat{\nu}(\pi_{i,j})=0$, and $\hat{\nu}(\pi_{i,j})=I$ implies $\hat{\mu}(\pi_{i,j})=I$. That is to say,
	\begin{equation}\label{Observation of order structure:0}
	\{(i,j)\in\mathbb{Z}_k\times\mathbb{Z}_k:\hat{\mu}(\pi_{i,j})=0\}\subset\{(i,j)\in\mathbb{Z}_k\times\mathbb{Z}_k:\hat{\nu}(\pi_{i,j})=0\},
	\end{equation}
	and 
	\begin{equation}\label{Observation of order structure:I}
	\{(i,j)\in\mathbb{Z}_k\times\mathbb{Z}_k:\hat{\nu}(\pi_{i,j})=I\}\subset\{(i,j)\in\mathbb{Z}_k\times\mathbb{Z}_k:\hat{\mu}(\pi_{i,j})=I\}.
	\end{equation}
	So from Remark \ref{remark of prop1} and Remark \ref{remark of prop2} it follows that
	\begin{enumerate}
		\item if $\mu\in\mathcal{I}_3(\mathcal{A}_k)$, then $\nu\notin\mathcal{I}_2(\mathcal{A}_k)$;
		\item if $\nu\in\mathcal{I}_1(\mathcal{A}_k)$, then $\mu\in\mathcal{I}_1(\mathcal{A}_k)$.
	\end{enumerate} 
	Hence we have only the following six cases:
	\begin{enumerate}
		\item [(1)] $\mu=h_{\Gamma^{(\mu)}},\nu=h_{\Gamma^{(\nu)}}\in\mathcal{I}_1(\mathcal{A}_k)$. Then $\mu\prec\nu$ if and only if 
		\begin{equation*}
		\alpha^{(\nu)}_{i,j}=\sum_{r,s\in \mathbb{Z}_k}\alpha^{(\mu)}_{i-r,j-s}\alpha^{(\nu)}_{r,s},~~ i,j\in \mathbb{Z}_k.
		\end{equation*}
		For $(i,j)\notin\Gamma^{(\nu)}$, we have $\alpha^{(\mu)}_{i-r,j-s}\alpha^{(\nu)}_{r,s}=0,\forall r,s\in\mathbb{Z}_k$. Choosing $(r,s)=(0,0)$, we have $\alpha^{(\mu)}_{i,j}=0$, i.e., $(i,j)\notin\Gamma^{(\mu)}$. So $\Gamma^{(\mu)}\prec\Gamma^{(\nu)}$.
		
		\item [(2)] $\mu=h_{\Gamma^{(\mu)}}\in\mathcal{I}_1(\mathcal{A}_k),\nu=h_{\Gamma^{(\nu)},l^{(\nu)}}\in\mathcal{I}_2(\mathcal{A}_k)$. Then $\mu\prec\nu$ if and only if 
		\begin{equation}\label{I_1<I_2 1}
		\alpha^{(\nu)}_{i,j}=\sum_{r,s\in \mathbb{Z}_k}\alpha^{(\mu)}_{i-r,j-s}\alpha^{(\nu)}_{r,s},~~ i,j\in \mathbb{Z}_k,
		\end{equation}
		and 
		\begin{equation}\label{I_1<I_2 2}
		\kappa^{(\nu)}_{r,r}=\sum_{i,j\in \mathbb{Z}_k}\alpha^{(\mu)}_{i,j}\kappa^{(\nu)}_{r+j,r+j},~~ r,s\in\mathbb{Z}_k.
		\end{equation}
		From a similar argument in (1), we have that \eqref{I_1<I_2 1} implies $\Gamma^{(\mu)}\prec\Gamma^{(\nu)}$. So $\Gamma^{(\mu)}=\mathbb{Z}_k\times q^{(\mu)}\mathbb{Z}_k$ with $q^{(\nu)}|q^{(\mu)}$. In this case \eqref{I_1<I_2 2} always holds.
		
		\item [(3)]
		$\mu=h_{\Gamma^{(\mu)}}\in\mathcal{I}_1(\mathcal{A}_k),\nu=h_{\Gamma^{(\nu)},l^{(\nu)},\tau^{(\nu)}}$. Then $\mu\prec\nu$ if and only if 
		\begin{equation}\label{I_1<I_3 1}
		\alpha^{(\nu)}_{i,j}=\sum_{r,s\in \mathbb{Z}_k}\alpha^{(\mu)}_{i-r,j-s}\alpha^{(\nu)}_{r,s},~~ i,j\in \mathbb{Z}_k,
		\end{equation}
		and 
		\begin{equation}\label{I_1<I_3 2}
		\kappa^{(\nu)}_{r,s}=\sum_{i,j\in \mathbb{Z}_k}\eta^{i(s-r)}\alpha^{(\mu)}_{i,j}\kappa^{(\nu)}_{r+j,s+j},~~ r,s\in\mathbb{Z}_k.
		\end{equation}
		From a similar argument in (2), we have $\Gamma^{(\mu)}=p^{(\mu)}\mathbb{Z}_k\times q^{(\mu)}\mathbb{Z}_k\prec \Gamma^{(\mu)}=p^{(\nu)}\mathbb{Z}_k\times q^{(\nu)}\mathbb{Z}_k$. Or equivalently, $p^{(\nu)}|p^{(\mu)}$ and $q^{(\nu)}|q^{(\mu)}$. In this case \eqref{I_1<I_3 2} always holds. 
		
		\item [(4)] $\mu=h_{\Gamma^{(\mu)},l^{(\mu)}},\nu=h_{\Gamma^{(\nu)},l^{(\nu)}}\in\mathcal{I}_2(\mathcal{A}_k)$. From Remark \ref{remark of prop2} and \eqref{Observation of order structure:0} it follows that $q^{(\nu)}|q^{(\mu)}$. So $\Gamma^{(\mu)}\prec\Gamma^{(\nu)}$. Moreover, $\hat{\mu}(\pi_{i,j})\prec \hat{\nu}(\pi_{i,j})$ requires that for all $\frac{k}{q^{(\nu)}}|j$
		\begin{equation*}
		\begin{psmallmatrix*}
		\frac{1}{2}&\frac{1}{2}\eta^{-jl^{(\nu)}}\\
		\frac{1}{2}\eta^{jl^{(\nu)}}&\frac{1}{2}
		\end{psmallmatrix*}
		=\begin{psmallmatrix*}
		\frac{1}{2}&\frac{1}{2}\eta^{-jl^{(\mu)}}\\
		\frac{1}{2}\eta^{jl^{(\mu)}}&\frac{1}{2}
		\end{psmallmatrix*}
		\begin{psmallmatrix*}
		\frac{1}{2}&\frac{1}{2}\eta^{-jl^{(\nu)}}\\
		\frac{1}{2}\eta^{jl^{(\nu)}}&\frac{1}{2}
		\end{psmallmatrix*},
		\end{equation*}
		which is equivalent to %for all $\frac{k}{q^{(\nu)}}|j$
		\begin{equation*}
		\begin{psmallmatrix*}
		\frac{1}{2}&\frac{1}{2}\eta^{-jl^{(\nu)}}\\
		\frac{1}{2}\eta^{jl^{(\nu)}}&\frac{1}{2}
		\end{psmallmatrix*}
		=\begin{psmallmatrix*}
		\frac{1}{4}+\frac{1}{4}\eta^{j(l^{(\nu)}-l^{(\mu)})}&\frac{1}{4}\eta^{-jl^{(\nu)}}(1+\eta^{j(l^{(\nu)}-l^{(\mu)})})\\
		\frac{1}{4}\eta^{jl^{(\nu)}}(1+\eta^{j(l^{(\mu)}-l^{(\nu)})})&\frac{1}{4}+\frac{1}{4}\eta^{j(l^{(\mu)}-l^{(\nu)})}
		\end{psmallmatrix*}.
		\end{equation*}
		That is to say,
		\begin{equation*}
		\eta^{j(l^{(\nu)}-l^{(\mu)})}=1 \text{ for all }\frac{k}{q^{(\nu)}}|j.
		\end{equation*}
		So $l^{(\mu)}\equiv l^{(\nu)}\mod q^{(\nu)}$.
		
		\item [(5)]$\mu=h_{\Gamma^{(\mu)},l^{(\mu)},\tau^{(\mu)}}\in\mathcal{I}_3(\mathcal{A}_k),\nu=h_{\Gamma^{(\nu)},l^{(\nu)}}\in\mathcal{I}_2(\mathcal{A}_k)$. Following a similar argument as above, it follows that $\mu\prec\nu$ if and only if $q^{(\nu)}|q^{(\mu)}$ and
		\begin{equation*}
		\begin{psmallmatrix*}
		\frac{1}{2}&\frac{1}{2}\tau^{(\nu)}_j\eta^{-jl^{(\nu)}}\\
		\frac{1}{2}\tau^{(\nu)}_j\eta^{jl^{(\nu)}}&\frac{1}{2}
		\end{psmallmatrix*}
		=\begin{psmallmatrix*}
		\frac{1}{2}&\frac{1}{2}\eta^{-jl^{(\mu)}}\\
		\frac{1}{2}\eta^{jl^{(\mu)}}&\frac{1}{2}
		\end{psmallmatrix*}
		\begin{psmallmatrix*}
		\frac{1}{2}&\frac{1}{2}\tau^{(\nu)}_j\eta^{-jl^{(\nu)}}\\
		\frac{1}{2}\tau^{(\nu)}_j\eta^{jl^{(\nu)}}&\frac{1}{2}
		\end{psmallmatrix*},
		\end{equation*}
		if and only if $\Gamma^{(\mu)}\prec\Gamma^{(\nu)}$ and 
		\begin{equation*}
		\tau^{(\nu)}_j=\eta^{j(l^{(\nu)}-l^{(\mu)})}\text{ for $j$ such that }\frac{k}{q^{(\nu)}}=p^{(\nu)}|j.
		\end{equation*}
		\item [(6)] $\mu=h_{\Gamma^{(\mu)},l^{(\mu)},\tau^{(\mu)}},\nu=h_{\Gamma^{(\nu)},l^{(\nu)},\tau^{(\nu)}}\in\mathcal{I}_3(\mathcal{A}_k)$. From Remark \ref{remark of prop2} and \eqref{Observation of order structure:0} it follows $p^{(\nu)}|p^{(\mu)}$ and $q^{(\nu)}|q^{(\mu)}$. Since $p^{(\mu)}q^{(\mu)}=p^{(\nu)}q^{(\nu)}=k$, we have  $p^{(\mu)}=p^{(\nu)},q^{(\mu)}=q^{(\nu)}$. Thus $\Gamma^{(\mu)}=\Gamma^{(\nu)}$. Moreover, 
		\begin{equation*}
		\begin{psmallmatrix*}
		\frac{1}{2}&\frac{1}{2}\tau^{(\nu)}_j\eta^{-jl^{(\nu)}}\\
		\frac{1}{2}\tau^{(\nu)}_j\eta^{jl^{(\nu)}}&\frac{1}{2}
		\end{psmallmatrix*}
		=\begin{psmallmatrix*}
		\frac{1}{2}&\frac{1}{2}\tau^{(\mu)}_j\eta^{-jl^{(\mu)}}\\
		\frac{1}{2}\tau^{(\mu)}_j\eta^{jl^{(\mu)}}&\frac{1}{2}
		\end{psmallmatrix*}
		\begin{psmallmatrix*}
		\frac{1}{2}&\frac{1}{2}\tau^{(\nu)}_j\eta^{-jl^{(\nu)}}\\
		\frac{1}{2}\tau^{(\nu)}_j\eta^{jl^{(\nu)}}&\frac{1}{2}
		\end{psmallmatrix*},
		\end{equation*}
		which is equivalent to
		\begin{equation*}
		\tau^{(\nu)}_j=\tau^{(\mu)}_j\eta^{j(l^{(\nu)}-l^{(\mu)})}\text{ for $j$ such that }p^{(\mu)}=p^{(\nu)}|j.
		\end{equation*}
	\end{enumerate}         
\end{proof}

We present here the order structure on $\text{Idem}(\mathcal{A}_k)$ for $k$ prime.

\begin{example}
    When $k$ is a prime number, $\mathbb{Z}_k\times\mathbb{Z}_k$ has one subgroup of order 1: $\Gamma_0=\{(0,0)\}$, $k+1$ subgroups of order $k$:
    \[
	\Gamma_+=\mathbb{Z}_k\times k\mathbb{Z}_k,~~
	\Gamma_-=k\mathbb{Z}_k\times\mathbb{Z}_k,~~
	\Gamma_i=\{j(1,i)=(j,ij):j\in\mathbb{Z}_k\},
	\]
	where $i=1,2,\cdots,k-1$, and 1 subgroup of order $k^2$: $\Gamma_{k}=\mathbb{Z}_k\times\mathbb{Z}_k$.
	
	Then Proposition \ref{Haar idempotent states} gives $k+3$ idempotent states: $h_+:=h_{\Gamma_+}$, $h_-:=h_{\Gamma_-}$, and $h_i:=h_{\Gamma_i},i=0,1,\cdots, k$, in which $h_0=\epsilon$ is the counit. The idempotent state in Theorem \ref{thm:non-Haar} (1) is the Haar state $h=h_{\mathcal{A}_k}$. By Theorem \ref{thm:non-Haar} (2), the Haar idempotent states of the form $h_{\Gamma,l}$ are $h_{+,l}:=h_{\Gamma_+,l}$ with $l\in\mathbb{Z}_k$. And the Theorem \ref{thm:non-Haar} (3) tells us that $h_{-,0,\tau}:=h_{\Gamma_-,0,\tau}$ are the only elements in $\mathcal{I}_3(\mathcal{A}_k)$, where $\tau$ verifies \eqref{a formula for K to be positive}. From Theorem \ref{Order structure} we can draw the Hasse diagram of the lattice $(\text{Idem}(\mathcal{A}_k),\prec)$ as:
	
	
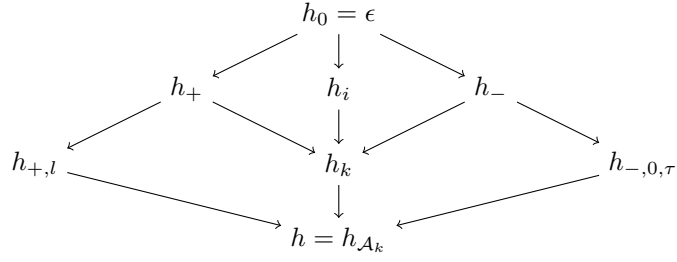
\begin{center}
		\begin{tikzpicture}
	\node (P0) at (0,0) {$h_0=\epsilon$};
	\node (P+) at (-2,-1) {$h_+$};
	\node (P-) at (2,-1) {$h_-$};
	\node (Pi) at (0,-1) {$h_i$};
	\node (Pk) at (0,-2) {$h_k$};
	\node (P+l) at (-4,-2) {$h_{+,l}$};
	\node (P-0) at (4,-2) {$h_{-,0,\tau}$};
	\node (P) at (0,-3) {$h=h_{\mathcal{A}_k}$};
	\draw
	(P0) edge[->] node[] {} (P+)
	(P0) edge[->] node[] {} (P-)
	(P0) edge[->] node[] {} (Pi)
	(P+) edge[->] node[] {} (P+l)
	(P+) edge[->] node[] {} (Pk)
	(P-) edge[->] node[] {} (P-0)
	(P-) edge[->] node[] {} (Pk)
	(Pi) edge[->] node[] {} (Pk)
	(P+l) edge[->] node[] {} (P)
	(P-0) edge[->] node[] {} (P)
	(Pk) edge[->] node[] {} (P);
	\end{tikzpicture}
	\end{center}
	where $i=1,2,\cdots,k-1$, $l\in\mathbb{Z}_k$, and $\tau$ satisfies \eqref{a formula for K to be positive}. When $k=2$, the Hasse diagram reads precisely as:
	
	
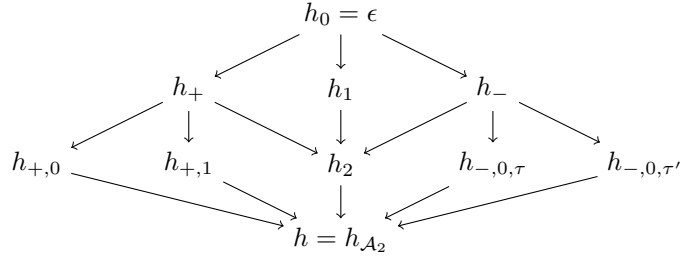
\begin{center}
		\begin{tikzpicture}
	\node (P0) at (0,0) {$h_0=\epsilon$};
	\node (P+) at (-2,-1) {$h_+$};
	\node (P-) at (2,-1) {$h_-$};
	\node (P1) at (0,-1) {$h_1$};
	\node (P2) at (0,-2) {$h_2$};
	\node (P+0) at (-4,-2) {$h_{+,0}$};
	\node (P+1) at (-2,-2) {$h_{+,1}$};
	\node (P-0) at (2,-2) {$h_{-,0,\tau}$};
	\node (P-1) at (4,-2) {$h_{-,0,\tau'}$};
	\node (P) at (0,-3) {$h=h_{\mathcal{A}_2}$};
	\draw
	(P0) edge[->=angle 30] node[] {} (P1)
	(P0) edge[->] node[] {} (P-)
	(P0) edge[->] node[] {} (P+)
	(P+) edge[->] node[] {} (P+0)
	(P+) edge[->] node[] {} (P+1)
	(P+) edge[->] node[] {} (P2)
	(P-) edge[->] node[] {} (P-0)
	(P-) edge[->] node[] {} (P-1)
	(P-) edge[->] node[] {} (P2)
	(P1) edge[->] node[] {} (P2)
	(P+0) edge[->] node[] {} (P)
	(P+1) edge[->] node[] {} (P)
	(P-0) edge[->] node[] {} (P)
	(P-1) edge[->] node[] {} (P)
	(P2) edge[->] node[] {} (P);
	\end{tikzpicture}
	\end{center}
	Here $\tau$ is the trivial one that satisfies \eqref{a formula for K to be positive}, i.e., $\tau_0=\tau_1=1$; and $\tau'$ is given through $\tau'_0=1$ and $\tau'_1=-1$, as we have mentioned in Remark \ref{remark of tau}. 
	One should compare this diagram with the one in \cite{FGquantumindependentincrementprocess} of eight-dimension Kac-Paljutkin quantum group.
	
	Note that the Hasse diagram for $k=2$ coincides with that of the lattice of subgroups of Dihedral group $D_4$ (with the partial order reversed). Indeed, from discussions in Section 1, $\mathcal{A}_2$ has 8 one-dimensional representations and no 2-dimensional ones. Hence it is cocommutative, and therefore equal to $C^*(\Gamma)$ for some classical group $\Gamma$. This group is nothing but $D_4$. To see this, take $x=\rho_{11}+\sigma_{11}$ and $y=\rho_{10}+\sigma_{10}$. It is not difficult to verify that $y^2=x^4=1$ and $yxy=x^{-1}$. Moreover, $x$ and $y$ do not admit any other independent relations. So $\mathcal{A}_2$ is the group algebra of $D_4$, and thus the lattice of idempotent states on $\mathcal{A}_2$ is nothing but the lattice of subgroups of $D_4$.
\end{example}

\section{Convergence of convolution powers of states on Sekine quantum groups}
We end this paper by saying a few words on the convergence of convolution powers of states on Sekine quantum groups. This is related to random walks on quantum groups \cite{FGquantumindependentincrementprocess,JPMcCarthy}.

Fix a state $\mu=\sum_{i,j\in \mathbb{Z}_k}\alpha_{i,j}\widetilde{d}_{i,j}+\sum_{r,s\in\mathbb{Z}_k}\kappa_{r,s}\widetilde{e}_{r,s}$ on $\mathcal{A}_k$. Then clearly $\{\mu^{\star n}\}_{n\geq1}$ converges if and only if $\{\hat{\mu}(\pi_{p,q})^{n}\}_{n\geq1}$ converges for all $p,q\in\mathbb{Z}_k$. The following proposition gives a sufficient condition that guarantees the convergence. 

\begin{prop}\label{random walk convergence}
	Let $\mu$ be as above. Then $\{\mu^{\star n}\}_{n\geq1}$ converges if $\alpha_{0,0}>0$.
\end{prop}

\begin{proof}
	Fix $p,q\in\mathbb{Z}_k$. Denote by $\lambda_1,\lambda_2$ the eigenvalues of $\hat{\mu}(\pi_{p,q})$. Let $\lambda\in\{\lambda_1,\lambda_2\}$, then we have 
	\begin{equation*}
	\left(\lambda-\mu(\rho_{p,q})\right)\left(\lambda-\mu(\rho_{p,-q})\right)
	=\mu(\sigma_{p,q})\mu(\sigma_{p,-q}).
	\end{equation*}
	Since $\pi_{p,q}$ is unitary, $\Vert\hat{\mu}(\pi_{p,q}) \Vert=\Vert(\text{id}_{\mathbb{M}_2(\mathbb{C})}\otimes\mu)(\pi_{p,q})\Vert\leq1$. Thus $|\lambda|\leq 1$.
	Note that 
	\begin{equation}\label{murholeqsumalpha}
	|\mu(\rho_{p,q})|
	=|\sum_{i,j\in\mathbb{Z}_k}\alpha_{i,j}\eta^{ip+jq}|
	\leq\sum_{i,j\in\mathbb{Z}_k}\alpha_{i,j},
	\end{equation}
	and from $K=[\kappa_{r,s}]_{r,s\in\mathbb{Z}_k}\geq0$ it follows that
	\begin{equation*}
	|\mu(\sigma_{p,q})|
	=|\sum_{l\in\mathbb{Z}_k}\kappa_{l,l+p}\eta^{ql}|
	\leq\sum_{l\in\mathbb{Z}_k}|\kappa_{l,l+p}|\leq\frac{1}{2}\sum_{l\in\mathbb{Z}_k}(\kappa_{l,l}+\kappa_{l+p,l+p})
	=\sum_{r\in\mathbb{Z}_k}\kappa_{r,r}.
	\end{equation*}
	If $|\lambda|=1$, then the equations above yield
	\begin{align*}
	(1-\sum_{i,j\in\mathbb{Z}_k}\alpha_{i,j})^2
	&\leq\left(|\lambda|-|\mu(\rho_{p,q})|\right)\left(|\lambda|-|\mu(\rho_{p,-q})|\right)\\
	&\leq|\left(\lambda-\mu(\rho_{p,q})\right)\left(\lambda-\mu(\rho_{p,-q})\right)|\\
	&=|\mu(\sigma_{p,q})\mu(\sigma_{p,-q})|\\
	&\leq(1-\sum_{i,j\in\mathbb{Z}_k}\alpha_{i,j})^2.
	\end{align*}
	Hence
	\begin{equation*}
	(1-\sum_{i,j\in\mathbb{Z}_k}\alpha_{i,j})^2
	=\left(1-|\mu(\rho_{p,q})|\right)\left(1-|\mu(\rho_{p,-q})|\right),
	\end{equation*}
	which gives $|\mu(\rho_{p,q})|=|\mu(\rho_{p,-q})|=\sum_{i,j\in\mathbb{Z}_k}\alpha_{i,j}$. Since $\alpha_{0,0}>0$, we have $\mu(\rho_{p,q})=\mu(\rho_{p,-q})=\sum_{i,j\in\mathbb{Z}_k}\alpha_{i,j}>0$. So $\lambda$ can be nothing but 1. That is to say, $\lambda\in\{z\in\mathbb{Z}:|z|<1\}\cup\{1\}$. Hence $\{\hat{\mu}(\pi_{p,q})^n\}_{n\geq1}$ converges if $\hat{\mu}(\pi_{p,q})$ is not similar to the Jordan normal form 
	$\begin{psmallmatrix*}
	1&1\\
	0&1
	\end{psmallmatrix*}.$
	
	If $\lambda_1=\lambda_2=1$, then from 
	\[
	2=\lambda_1+\lambda_2
	=\mu(\rho_{p,q})+\mu(\rho_{p,-q})
	\leq 2\sum_{i,j\in\mathbb{Z}_k}\alpha_{i,j}
	\leq2
	\] 
	it follows that $\mu(\rho_{p,q})=\mu(\rho_{p,-q})=\sum_{i,j\in\mathbb{Z}_k}\alpha_{i,j}=1$. Thus $\mu(\sigma_{p,q})=\mu(\sigma_{p,-q})=\sum_{r\in\mathbb{Z}_k}\kappa_{r,r}=0$. That is to say, $\hat{\mu}(\pi_{p,q})$ equals identity, not similar to the Jordan normal form as above, which finishes the proof.
\end{proof}

\begin{rem}
	In the proof we only used $\alpha_{0,0}>0$ to deduce $\mu(\rho_{p,q})=\mu(\rho_{p,-q})=\sum_{i,j\in\mathbb{Z}_k}\alpha_{i,j}$ from $|\mu(\rho_{p,q})|=|\mu(\rho_{p,-q})|=\sum_{i,j\in\mathbb{Z}_k}\alpha_{i,j}$. Recalling \eqref{murholeqsumalpha}, to make sure that $\{\mu^{\star n}\}_{n\geq 1}$ converges when $\alpha_{0,0}=0$, it suffices to assume $\sharp\{\eta^{ip+jq}:\alpha_{i,j}\ne 0\}\geq 2$ for all $p,q\in\mathbb{Z}_k$.
\end{rem}

\subsection*{Acknowledgment}
Part of this work was done during a visit to Seoul National University. The author would like to thank Hun Hee Lee and Xiao Xiong for their hospitality. He would also like to thank Adam Skalski and Uwe Franz for many useful comments and pointing out some mistakes in an earlier version of the paper. The research was partially supported by the NCN (National Centre of Science) grant 2014/14/E/ST1/00525, the French ``Investissements d'Avenir" program, project ISITE-BFC (contract ANR-15-IDEX-03) and NSFC No. 11431011.

\end{document}